\documentclass[11pt]{amsart}
\usepackage{lmodern} 
\usepackage[T1]{fontenc}

\usepackage{amsfonts,dsfont}
\usepackage{pstricks}
\usepackage{epsfig}
\usepackage{amsmath}
\usepackage{amssymb}
\theoremstyle{plain}
\newtheorem{theorem}{Theorem} [section]
\newtheorem*{theorem*}{Theorem}
\newtheorem{proposition} [theorem] {Proposition}

\newtheorem{lemma} [theorem] {Lemma}

\theoremstyle{definition}

\theoremstyle{remark}
\newtheorem{remark} [theorem] {Remark}

\def \r{\mbox{${\mathbb R}$}}
\def \h{\mathbb H}
\def \e{\mbox{${\mathbb E}$}}
\def \s{{\mathbb S}}
\def \l{{\mathbb L}}

\DeclareMathOperator{\grad}{grad}
\DeclareMathOperator{\trace}{trace}
\DeclareMathOperator{\cst}{constant}
\DeclareMathOperator{\Div}{div}
\DeclareMathOperator{\ricci}{Ricci}

\begin{document}

\begin{abstract}
We introduce the notion of {\em biconservative} hypersurfaces, that is hypersurfaces with conservative  {\it stress-energy}  
tensor with respect to the bienergy. We give the (local) classification of biconservative surfaces in 3-dimensional space forms.
\end{abstract}

\title{Surfaces in three-dimensional space forms with divergence-free stress-bienergy tensor}

\author{R.~Caddeo}
\address{Universit\`a degli Studi di Cagliari\\
Dipartimento di Matematica e Informatica\\
Via Ospedale 72\\
09124 Cagliari, ITALIA}
\email{caddeo@unica.it}
\email{montaldo@unica.it}
\email{piu@unica.it}
\author{S.~Montaldo}

\author{C.~Oniciuc}
\address{Faculty of Mathematics\\ ``Al.I. Cuza'' University of Iasi\\
Bd. Carol I no. 11 \\
700506 Iasi, ROMANIA}
\email{oniciucc@uaic.ro}

\author{P.~Piu}

\thanks{The third author was supported  by a grant of the Romanian National Authority for
Scientific Research, CNCS -- UEFISCDI, project number PN-II-RU-TE-2011-3-0108; and by Regione Autonoma della Sardegna, Visiting Professor Program.}


\maketitle

\section{Introduction}

A hypersurface $M^m$ in an $(m+1)$-dimensional Riemannian manifold $N^{m+1}$ is called {\it biconservative} if
\begin{equation}\label{eq-def-biconservative-space}
2 A(\grad f)+ f \grad f=2 f \ricci^N(\eta)^{\top}\,,
\end{equation}
where $A$ is the shape operator, $f=\trace A$ is the mean curvature function and $\ricci^N(\eta)^{\top}$ is the tangent  component of the Ricci curvature of $N$ in the direction of the unit normal $\eta$ of $M$ in $N$.

The name biconservative, as we shall describe in Section~\ref{sec:stress-energy-tensor}, comes from the fact that condition \eqref{eq-def-biconservative-space} is equivalent to the conservativeness   of a certain {\it stress-energy}  tensor $S_2$, that is $\Div S_2=0$ if and only if the  hypersurface  is biconservative. The tensor $S_2$ is associated to the bienergy functional. In general, a submanifold is called {\em biconservative} if $\Div S_2=0$.

Moreover, the class of biconservative submanifolds includes that of biharmonic  submanifolds, which have been of large interest in the last decade (see, for example, \cite{BMO12,BMO10a,BMO10,BMO08,CI91,O10,OW11}). Biharmonic submanifolds are characterized by the vanishing of the bitension field and they represent a generalization of harmonic (minimal) submanifolds. In fact, as detailed in Section~\ref{sec:stress-energy-tensor},  a submanifold is biconservative if the tangent part of the bitension field vanishes.
It is worth to point out that, thinking at the energy functional instead of the bienergy functional, the notion of {\em conservative} submanifolds  is not useful as all submanifolds are conservative (see Remark~\ref{remark:conservative}).
We also would like to point out that submanifolds with vanishing  tangent part of the bitension field have been considered by Sasahara in \cite{S2012} where he studied certain 3-dimensional submanifolds in $\r^6$. 

In this paper we consider biconservative surfaces in a 3-dimensional space form $N^3(c)$ of constant sectional curvature $c$. In this case \eqref{eq-def-biconservative-space} becomes
\begin{equation}\label{eq-def-biconservative-space-form}
2 A(\grad f)+ f \grad f=0\,.
\end{equation}
From \eqref{eq-def-biconservative-space-form} we see that CMC surfaces, i.e. surfaces with constant mean curvature, in space forms are biconservative.  Thus our interest will be on NON CMC biconservative surfaces. 

As a general fact, we first prove that the mean curvature function $f$ of a biconservative surface in a 3-dimensional space form satisfies the following PDE
$$
f \Delta f+|\grad f|^2-\frac{16}{9} K(K-c)=0\,,
$$
where $K$ denotes the Gauss curvature of the surface, while $\Delta$ is the Laplace-Beltrami operator on $M$.

Then the paper is completely devoted to the local classification of biconservative surfaces in 3-dimensional space forms. This is done in three sections where we examine, separately, the cases of: surfaces in the 3-dimensional euclidean space; surfaces in the  3-dimensional sphere; surfaces in the  3-dimensional hyperbolic space.
\vspace{2mm}

For biconservative surfaces in $\r^3$, we shall reprove a result of Hasanis and Vlachos contained in \cite{HV95}, where they call $H$-surfaces the biconservative surfaces.\vspace{2mm}

\noindent {\bf Theorem}~{\bf \ref{teo:r3general}.}
{\em Let $M^2$ be a biconservative surface in $\r^3$ with $f(p)>0$ and $\grad f(p)\neq 0$ for any $p\in M$. Then, locally, $M^2$ is a surface of revolution.}
\vspace{2mm}

In fact, we give the explicit parametrization of the profile curve of a biconservative surface of  revolution (see Proposition~\ref{pro-surf-revolution}), which is not in \cite{HV95}. In their paper, the authors also studied the case of biconservative hypersurfaces in $\r^4$ obtaining a similar result to Theorem~\ref{teo:r3general}.

 
Our approach is slightly different and allow us to go further and classify the
biconservative surfaces in $\s^3$ and in $\h^3$. Moreover,  the notion of biconservative submanifolds is more general than the notion of $H$-hypersurfaces in $\r^n$.
\vspace{2mm}

Considering $\s^3$ as a submanifold of  $\r^4$, the biconservative surfaces in $\s^3$ are characterized by the following\vspace{2mm}

\noindent {\bf Theorem}~{\bf \ref{teo:class-bicons-s3}.}
{\em Let $M^2$ be a biconservative surface in $\s^3$ with $f(p)>0$ and $\grad f(p)\neq 0$ at any point $p\in M$. Then, 
 locally, $M^2\subset \r^4$ can be parametrized by
$$
X_C(u,v)=\sigma(u)+\frac{4}{3 \sqrt{C} k(u)^{3/4}}\big(C_1 (\cos v-1) +C_2 \sin v \big),
$$
where  $C$ is a positive constant of integration, $C_1,C_2\in\r^4$ are two constant orthonormal vectors such that
$$
\langle \sigma(u), C_1\rangle =\frac{4}{3 \sqrt{C} k(u)^{3/4}}\,,\quad \langle \sigma(u), C_2\rangle=0\,,
$$
while $\sigma=\sigma(u)$ is a curve lying in the totally geodesic $\s^2=\s^3\cap\Pi$ ($\Pi$ the linear hyperspace of $\r^4$ orthogonal to $C_2$), whose geodesic  curvature $k=k(u)$ is a positive non constant solution of
the following ODE
$$
k'' k =\frac{7}{4} (k')^2+\frac{4}{3} k^2-4 k^4\,.
$$
}
\vspace{2mm}

Geometrically Theorem~\ref{teo:class-bicons-s3} means that, locally, the surface $M^2$ is  given by a family of circles of $\r^4$, passing through the curve $\sigma$,  and belonging to a pencil of planes which are parallel to the linear space spanned by $C_1$ and $C_2$. Now, these circles must be the intersection of the pencil with the sphere $\s^3$. Let $G$ be the 1-parameter group of isometries of $\r^4$ generated by the Killing vector field
$$
T=\langle {\mathbf r}, C_2\rangle C_1+\langle  {\mathbf r}, C_1\rangle C_2\,,
$$ 
where ${\mathbf r}$ represents  the position vector of a point in $\r^4$.
Then $G$ acts also on $\s^3$ by isometries and it can be identified with the group $SO(2)$. Since the orbits
of $G$ are circles of $\s^3$ we deduce that $X(u,v)$, in Theorem~\ref{teo:class-bicons-s3}, describes an $SO(2)$ invariant 
surface of $\s^3$ obtained by the action of $G$ on the curve $\sigma$. Moreover, as we shall explain in Remark~\ref{remark-existences3}, 
there exist solutions of the ODE in Theorem~\ref{teo:class-bicons-s3} for the corresponding profile curve $\sigma$.  Although we are not able to give explicit solutions 
for $\sigma$, as we have done for the biconservative surfaces in $\r^3$, using Mathematica we give a plot of a numerical solution of the ODE in Theorem~\ref{teo:class-bicons-s3}, which describes the behavior of the curvature of $\sigma$. 
  \vspace{2mm}

Let consider the following model for the hyperbolic space
$$
\h^3 = \{(x_1,x_2,x_3,x_4)\in\l^4\,:\, x_1^2 +x_2^2 +x_3^2 -x_4^2 = -1,\, x_4 > 0\},
$$
where  $\l^4$ is the 4-dimensional Lorentz-Minkowski space. Then we have the following description of biconservative surfaces in $\h^3$. \vspace{2mm}

\noindent {\bf Theorem}~{\bf \ref{teo:class-bicons-h3}.}
{\em Let $M^2$ be a biconservative surface in $\h^3$ with $f(p)>0$ and $\grad f(p)\neq 0$ at any point $p\in M$. Put
$W={9|\grad f|^2}/({16f^2})+{9f^2}/{4}-1$.  Then, 
 locally, $M^2\subset \l^4$ can be parametrized by:
 \begin{enumerate}
 \item[(a)] if $W>0$
$$
X_C(u,v)=\sigma(u)+\frac{4}{3 \sqrt{C} k(u)^{3/4}}\big(C_1 (\cos v-1) +C_2 \sin v \big),
$$
where  $C$ is a positive constant of integration, $C_1,C_2\in\l^4$ are two constant vectors such that
$$
\langle C_i,C_j\rangle =\delta_{ij}\,,\quad \langle \sigma(u), C_1\rangle =\frac{4}{3 \sqrt{C} k(u)^{3/4}}\,,\quad \langle \sigma(u), C_2\rangle=0\,,
$$
while $\sigma=\sigma(u)$ is a curve lying in the totally geodesic $\h^2=\h^3\cap\Pi$ ($\Pi$ the linear hyperspace of $\l^4$ defined by $\langle{\mathbf r},C_2\rangle=0$), whose geodesic  curvature $k=k(u)$ is a positive non constant solution of
the following ODE
$$
k'' k =\frac{7}{4} (k')^2-\frac{4}{3} k^2-4 k^4\,.
$$
 \item[(b)]  if $W<0$
 $$
X_C(u,v)=\sigma(u)+\frac{4}{3 \sqrt{-C} k(u)^{3/4}}\big(C_1 (e^v-1) +C_2 (e^{-v} -1)\big),
$$
where  $C$ is a negative constant of integration, $C_1,C_2\in\l^4$ are two constant vectors such that
$$
\langle C_i,C_i\rangle =0\,,\quad\langle C_1,C_2\rangle =-1\,,\quad \langle \sigma(u), C_1\rangle =\langle \sigma(u), C_2\rangle=-\frac{2\sqrt{2}}{3 \sqrt{-C} k(u)^{3/4}}\,,
$$
while $\sigma=\sigma(u)$ is a curve lying in the totally geodesic $\h^2=\h^3\cap\Pi$ ($\Pi$ the linear hyperspace of $\l^4$ orthogonal to $C_1-C_2$), whose geodesic  curvature $k=k(u)$ is a positive non constant solution of the same ODE in {\rm (a)}.
\end{enumerate}
}
\vspace{2mm}
We note that a surface in a 3-dimensional space form for which both tangent and normal part of its bitension field vanish, i.e. a biharmonic surface, must be CMC (see \cite{RCSMCO2,C91}). Therefore, the assumption  that only the tangent part of the bitension field vanishes does not imply that the surface is CMC.
\vspace{2mm}

{\bf Conventions.}
Throughout this paper all manifolds, metrics, maps are assumed to be smooth, i.e. 
of class $C^\infty$. All manifolds are assumed to be connected. The following sign conventions are used
$$
\Delta^\varphi V=-\trace\nabla^2 V\,,\qquad R^N(X,Y)=[\nabla_X,\nabla_Y]-\nabla_{[X,Y]},
$$
where $V\in C(\varphi^{-1}(TN))$ and $X,Y\in C(TN)$.

By a {\it submanifold} $M$ in a Riemannian manifold $(N,h)$ we mean an isometric immersion $\varphi :M\to (N,h)$.
\vspace{2mm}

{\bf Acknowledgement.} The authors would like to thank Ye-Lin Ou  for reading a first draft of the paper and making some helpful suggestions.

\section{Biharmonic maps and the stress-energy tensor}\label{sec:stress-energy-tensor}

As described by Hilbert in~\cite{DH}, the {\it stress-energy}
tensor associated to a variational problem is a symmetric
$2$-covariant tensor $S$ conservative at critical points,
i.e. with $\Div S=0$.

In the context of harmonic maps $\varphi:(M,g)\to (N,h)$ between two Riemannian manifolds, that by definition are critical points of the
energy
$$
E(\varphi)=\frac{1}{2}\int_M\vert d\varphi\vert^2 \ v_g,
$$
the stress-energy tensor was studied in detail by
Baird and Eells in~\cite{PBJE} and Sanini in \cite{AS}. Indeed, the Euler-Lagrange
equation associated to the energy is equivalent to the vanishing of the tension
field $\tau(\varphi)=\trace\nabla d\varphi$ (see \cite{ES}), and the tensor
$$
S=\frac{1}{2}\vert d\varphi\vert^2 g - \varphi^{\ast}h
$$
satisfies $\Div S=-\langle\tau(\varphi),d\varphi\rangle$. Therefore, $\Div S=0$ when the map is harmonic.

\begin{remark}\label{remark:conservative}
We point out that, in the case of isometric immersions, the condition $\Div S=0$ is always satisfied,
 since $\tau(\varphi)$ is normal.
\end{remark}

A natural generalization of harmonic
maps, first proposed in \cite{EL}, can be obtained considering  the  {\it
bienergy} of $\varphi:(M,g)\to (N,h)$ which is defined by
$$
E_2(\varphi)=\frac{1}{2}\int_M\vert\tau(\varphi)\vert^2 \ v_g.
$$
The map $\varphi$ is {\it biharmonic} if it is a critical point of $E_2$ or,
equivalently, if it satisfies the associated Euler-Lagrange
equation
$$
\tau_2(\varphi)=-\Delta\tau(\varphi) -\trace R^N(d\varphi,\tau(\varphi))d\varphi
= 0.
$$
The study of the stress-energy tensor for the
bienergy was initiated  in \cite{GYJ2} and afterwards developed in \cite{LMO}. Its expression is \begin{eqnarray*}
S_2(X,Y)&=&\frac{1}{2}\vert\tau(\varphi)\vert^2\langle X,Y\rangle+
\langle d\varphi,\nabla\tau(\varphi)\rangle \langle X,Y\rangle \\
\nonumber && -\langle d\varphi(X), \nabla_Y\tau(\varphi)\rangle-\langle
d\varphi(Y), \nabla_X\tau(\varphi)\rangle,
\end{eqnarray*}
and it satisfies the condition
\begin{equation}\label{eq:2-stress-condition}
\Div S_2=-\langle\tau_2(\varphi),d\varphi\rangle,
\end{equation}
thus conforming to the principle of a  stress-energy tensor for the
bienergy.

If  $\varphi:(M,g)\to (N,h)$ is an isometric immersion then \eqref{eq:2-stress-condition} becomes
$$
\Div S_2=-\tau_2(\varphi)^{\top}.
$$
This means that isometric immersions with $\Div S_2=0$ correspond to immersions with vanishing tangent part of the corresponding bitension field.
The decomposition of the bitension field with respect to its normal and
tangent components was obtained with contributions of \cite{BMO12,C84,LM08,O02,O10} and for hypersurfaces it can be summarized in the following theorem.

\begin{theorem}\label{th: bih subm N}
Let  $\varphi:M^m\to N^{m+1}$ be an isometric immersion with mean curvature vector field $H=f\eta$. Then, $\varphi$ is biharmonic if and only if the normal and the tangent components of $\tau_2(\varphi)$  vanish,  i.e. respectively

\begin{subequations}
\begin{equation}\label{eq: caract_bih_normal}
\Delta {f}-f |A|^2+ f\ricci^N(\eta,\eta)=0,
\end{equation}
and
\begin{eqnarray}\label{eq: caract_bih_tangent}
2 A(\grad f)+  f \grad f-2 f \ricci^N(\eta)^{\top}=0
\end{eqnarray}
\end{subequations}
where $A$ is the shape operator, $f=\trace A$ is the mean curvature function and $\ricci^N(\eta)^{\top}$ is the tangent  component of the Ricci curvature of $N$ in the direction of the unit normal $\eta$ of $M$ in $N$.
\end{theorem}

Finally, from \eqref{eq: caract_bih_tangent}, an isometric immersion $\varphi:M^m\to N^{m+1}$  satisfies $\Div S_2=0$, i.e. it is biconservative, if and only if
$$
2 A(\grad f)+  f \grad f-2 f \ricci^N(\eta)^{\top}=0
$$
which is Equation~\eqref{eq-def-biconservative-space} given in the introduction.

\section{Biconservative surfaces in the 3-dimensional space forms}\label{sec-space-forms}

In this section we consider the case of biconservative surfaces $M^2$ in a 3-dimensional  space form $N^3(c)$ of sectional curvature $c$. 
In this setting \eqref{eq-def-biconservative-space}  becomes

\begin{equation}\label{eq:biconservative-r3}
A(\grad f)=-\frac{f}{2} \grad f\,.
\end{equation}

If $M^2$ is a CMC surface, that is $f=\cst$, then $\grad f=0$ and \eqref{eq:biconservative-r3}
is automatically satisfied. Thus biconservative surfaces include the class CMC surfaces whether compact or not.

We now assume that $\grad f\neq 0$ at a point $p\in M$ and, therefore, there exists a
neighbourhood  $U$ of $p$ such that $\grad f\neq 0$  at any point of $U$.
On the set $U$ we can define an orthonormal frame $\{X_1,X_2\}$ of vector fields by
\begin{equation}\label{eq:def-x1-x2}
X_1=\frac{\grad f}{|\grad f|}\,,\quad X_2\perp X_1\,,\quad |X_2|=1\,.
\end{equation}
From \eqref{eq:biconservative-r3} we have
$$
A(X_1)=-\frac{f}{2}X_1\,,
$$
thus $X_1$ is a principal direction corresponding to the principal curvature $\lambda_1=-f/2$.
Since $X_2\perp X_1$, $X_2$ is a principal direction with eigenvalue $\lambda_2$ such that
$$
f=\trace A=\lambda_1+\lambda_2=-\frac{f}{2}+\lambda_2
$$
and therefore $\lambda_2=3f/2$. From this, using the Weingarten equation, we immediately see that the Gauss curvature of the surface is
\begin{equation}\label{eq:gauss-biconservative}
K=\det A+c=- 3f^2/4+c
\end{equation}
and the norm of the shape operator is $|A|^2=5f^2/2$. Moreover, by the definition of $X_1$, we obtain
$$
(X_1 f) X_1=\langle \grad f, X_1 \rangle X_1 =\grad f.
$$
Thus,
$$
\grad f=(X_1 f)X_1 +(X_2 f) X_2=\grad f + (X_2 f) X_2,
$$
which implies that
\begin{equation}\label{eq:x2f0}
X_2 f=0.
\end{equation}

We are now in the right position to state the main result of this section.

\begin{theorem}\label{teo:bi-conservative-laplacian}
Let $M^2$ be a biconservative surface in $N^3(c)$ which is not CMC. Then, there exists an open subset $U$ of $M$, such that the restriction of $f$ in $U$ satisfies the following equations
\begin{equation}\label{eq:gauss-biconservative-neg}
K=\det A+c=- 3f^2/4+c
\end{equation}
and
\begin{equation}\label{eq:f-bi-conservative-local}
f \Delta f+|\grad f|^2-\frac{16}{9} K(K-c)=0,
\end{equation}
where $\Delta$ is the Laplace-Beltrami operator on $M$.
\end{theorem}
\begin{proof}
Since $M^2$ is not CMC, there exists a point $p$ with $\grad f(p)\neq 0$. Thus $\grad f \neq 0$ in a neighborhood $V$ of $p$. Now, since $f$ cannot be zero for all $q\in V$, there exists an open set $U\subset V$ with $f(q)\neq 0$ for all $q\in U$. Let us define on $U$ the local orthonormal frame $\{X_1,X_2\}$ as
in \eqref{eq:def-x1-x2} and let  $\{\omega^1,\omega^2\}$ be the  dual $1$-forms of $\{X_1,X_2\}$
with $\omega^j_i$  the connection $1$-forms given by $\nabla X_i=\omega^j_iX_j$.
Since $f\neq0$  on $U$, we can assume that $f>0$  on $U$.  

Equation  \eqref{eq:gauss-biconservative-neg} is just \eqref{eq:gauss-biconservative}. We shall prove \eqref{eq:f-bi-conservative-local}.

Since $A(X_1)=-(f/2) X_1$ and $A(X_2)=(3f/2) X_2$, from the Codazzi equation
$$
\nabla_{X_1}A(X_2)-\nabla_{X_2}A(X_1)=A([X_1,X_2])
$$
we obtain
$$
\left(4f \omega^1_2(X_1)+X_2f \right)X_1+
\left(3X_1f +4f \omega^2_1(X_2)\right)X_2=0.
$$
Since $X_2f=0$ and $f(p)\neq 0$ for all $p\in U$, we deduce that
\begin{equation}\label{eq:w12x1-w12x2}
\begin{cases}
\omega^1_2(X_1)=0\\
\\
\omega^1_2(X_2)=\dfrac{3}{4} \dfrac{X_1f}{f}\,.
\end{cases}
\end{equation}
Next, using \eqref{eq:w12x1-w12x2}, the Gauss curvature of $M^2$ is
$$
K=\langle R(X_1,X_2)X_2,X_1\rangle=X_1(\omega^1_2(X_2))-(\omega^1_2(X_2))^2,
$$
that, together with \eqref{eq:gauss-biconservative}, gives
$$
-\frac{3f^2}{4}+c=X_1(\omega^1_2(X_2))-(\omega^1_2(X_2))^2
$$
which is equivalent, taking into account \eqref{eq:w12x1-w12x2}, to
\begin{equation}\label{eq:f-prime-equation}
(X_1X_1f) f = \frac{7}{4} (X_1f)^2+ \frac{4c}{3}f^2-f^4\,.
\end{equation}
Now, a straightforward computation gives
$$
-\Delta f= X_1X_1f - \frac{3}{4f} (X_1f)^2\,,
$$
that, substituted in \eqref{eq:f-prime-equation}, taking into account \eqref{eq:gauss-biconservative},
yields the desired equation
$$
f \Delta f+|\grad f|^2-\frac{16}{9} K(K-c)=0.
$$
\end{proof}
\section{Biconservative surfaces in $\r^3$}\label{sec:bi-r3}
We shall now consider the case of biconservative surfaces in $\r^3$. 
We start our study investigating in detail the case of surfaces of revolution. Without loss of generality we can assume that the surface is (locally) parametrised by
\begin{equation}\label{eq:rotation-biconservative}
X(u,v)=(\rho(u) \cos v, \rho(u) \sin v, u)
\end{equation}
where the real valued function $\rho$ is assumed to be positive. The induced metric is
$ds^2=(1+\rho'^2)du^2+\rho^2dv^2$, and a routine calculation gives
$$
A=\left(\begin{array}{cc}
-\dfrac{\rho''}{(1+\rho'^2)^{3/2}}&0\\
0& \dfrac{1}{\rho(1+\rho'^2)^{1/2}}
\end{array}\right)\,.
$$
Thus
$$
f= \dfrac{1}{(1+\rho'^2)^{1/2}}\left(\frac{1}{\rho}-\dfrac{\rho''}{(1+\rho'^2)}\right)\,,
$$
and
$$
\grad f=\dfrac{1}{(1+\rho'^2)} f' \frac{\partial}{\partial u}.
$$
Then \eqref{eq:biconservative-r3} becomes
\begin{equation}\label{eq:biconservative-rho-revolution}
\frac{f'}{2(1+\rho'^2)^{3/2}} \left(\dfrac{3\rho''}{1+\rho'^2} - \dfrac{1}{\rho} \right)= 0\,.
\end{equation}

\begin{proposition}\label{pro-surf-revolution}
Let $M^2$ be a biconservative surface of revolution in $\r^3$ with non constant mean curvature. Then, locally, the surface can be parametrized by
$$
X_C(\rho,v)=(\rho \cos v, \rho \sin v, u(\rho))
$$
where 
$$
u(\rho)=\frac{3}{2 C}\left(\rho^{1/3} \sqrt{C\rho^{2/3}-1}+\frac{1}{\sqrt{C}} \ln\left[2(C \rho^{1/3} + \sqrt{C^2\rho^{2/3}-C})\right]\right)\,,$$
with $C$  a positive constant and $\rho\in(C^{-3/2},\infty)$. The parametrization $X_C$ consists of a family of biconservative surfaces of revolution any two of which  are not locally isometric.
\end{proposition}
\begin{proof}
If $f$ is not constant, then from \eqref{eq:biconservative-rho-revolution} we must have that $\rho$ is a solution of the following ODE
\begin{equation}\label{eq:ode-biconservative}
3\rho\, \rho''=1+(\rho')^2\,.
\end{equation}
We shall now integrate \eqref{eq:ode-biconservative}. Using the change of variables $y=\rho'^2$ we get
$$
3 \dfrac{dy}{1+y}=2\dfrac{d\rho}{\rho}\,.
$$
Integration yields 
$$
\rho'^2=C \rho^{2/3}-1\,,
$$
where $C$ is a positive constant. Thus 
$$
\frac{d\rho}{\sqrt{C \rho^{2/3}-1}}=\pm du\,.
$$
Now, using the change of variable $y=\rho^{1/3}$,  we obtain
$$
\frac{3 y^2}{\sqrt{C y^2-1}}dy= \pm du\,.
$$
The latter equation can be integrated and, up to a symmetry with respect to the $xy$-plane, followed by a translation along the vertical $z$-axis, gives the following solution
$$
u=u(\rho)=\frac{3}{2 C}\left(\rho^{1/3} \sqrt{C\rho^{2/3}-1}+\frac{1}{\sqrt{C}} \ln\left[2(C \rho^{1/3} + \sqrt{C^2\rho^{2/3}-C})\right]\right)\,,
$$
where $\rho\in(C^{-3/2},\infty)$. Since the derivative of $u(\rho)$ is
$$
u'(\rho)=\frac{1}{\sqrt{C\rho^{2/3}-1}}
$$
we deduce that $u(\rho)$ is invertible for $\rho\in(C^{-3/2},\infty)$ and its inverse function produces the desired 
solution of \eqref{eq:ode-biconservative}. For a plot of the function $u(\rho)$ see Figure~\ref{fig:ploturho}.
\end{proof}

\begin{remark}
If we denote by $\sigma(u)=(\rho(u),0,u)$ the profile curve of the surface described in Proposition~\ref{pro-surf-revolution} and we reparametrize it by arclength, then its curvature function $k$ satisfies the ODE
$$
k k''=\frac{7}{4} (k')^2-4 k^4\,.
$$
Moreover,  the Gauss curvature and mean curvature functions of the surface are
$$
K(\rho,v)=-\frac{1}{3 C \rho^{8/3}}\,,\quad f(\rho,v)=\frac{2}{3 \sqrt{C} \rho^{4/3}}\,.
$$
It is worth remarking that $f$ is non constant (as assumed in the Proposition~\ref{pro-surf-revolution}) and that the values of $K$ and $f$ are in accord with \eqref{eq:gauss-biconservative}.
\end{remark}
\begin{figure}[h!]
\begin{pspicture}(-5,-.5)(5,6)
\psset{xunit=1cm,yunit=1cm,linewidth=.03,arrowscale=2}
\put(-3,.5){\includegraphics[width=5.5cm]{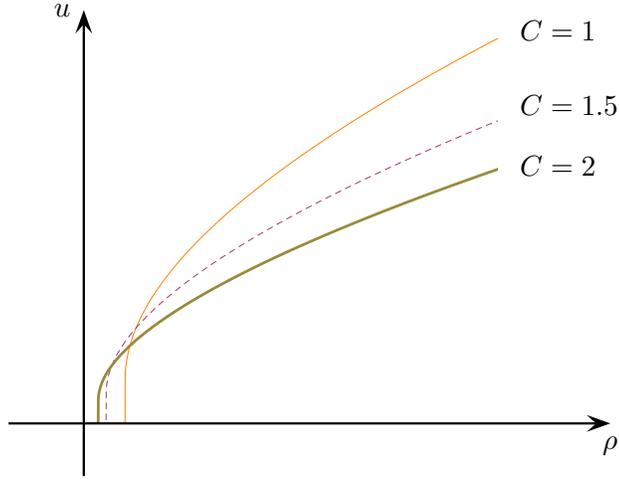}}
\psline{->}(-4,.5)(4,.5)
\psline{->}(-3,-.2)(-3,6)
\uput[270](4,.5){$\rho$}
\uput[180](-3,6){$u$}
\put(2.8,5.6){$C=1$}
\put(2.8,4.6){$C=1.5$}
\put(2.8,3.8){$C=2$}
\end{pspicture}
\caption{\label{fig:ploturho}Plots of the function $u(\rho)$ for $C=1$, $C=1.5$ and $C=2$.}

\end{figure}
\subsection{The general case}
We shall now prove that, essentially, the family described in Proposition~\ref{pro-surf-revolution} gives, locally,  all non CMC biconservative surfaces. To achieve this we assume that $f$ is strictly positive and that $\grad f\neq 0$ at any point. We define the local orthonormal frame $\{X_1,X_2\}$ as
in \eqref{eq:def-x1-x2} and from the calculations in the proof of Theorem~\ref{teo:bi-conservative-laplacian}  we have
\begin{equation}\label{eq:connection-M}
\left\{\begin{array}{ll}
\nabla_{X_1}X_1=0,&\quad \nabla_{X_1}X_2=0,\\
&\\
\nabla_{X_2}X_1=-\dfrac{3(X_1f)}{4f}X_2,&\quad \nabla_{X_2}X_2=\dfrac{3(X_1f)}{4f}X_1\,.
\end{array}
\right.
\end{equation}
Let $\eta$ be a unit vector field normal to the surface $M$. Then, if we denote by $\overline{\nabla}$ the connection of $\r^3$,  a straightforward computation gives
\begin{equation}\label{eq:connection-r3}
\left\{\begin{array}{ll}
\overline{\nabla}_{X_1}X_1=-\dfrac{f}{2}\eta,&\quad \overline{\nabla}_{X_1}X_2=0,\\
&\\
\overline{\nabla}_{X_2}X_1=-\dfrac{3(X_1f)}{4f}X_2,&\quad \overline{\nabla}_{X_2}X_2=\dfrac{3(X_1f)}{4f}X_1+\dfrac{3f}{2}\eta,\\
&\\
 \overline{\nabla}_{X_1}\eta=\dfrac{f}{2}X_1,&\quad \overline{\nabla}_{X_2}\eta=-\dfrac{3f}{2}X_2\,.
\end{array}
\right.
\end{equation}
Put
\begin{equation}\label{eq-defxi}
\kappa_2\, \xi=\dfrac{3(X_1f)}{4f}X_1+\dfrac{3f}{2}\eta=\overline{\nabla}_{X_2}X_2
\end{equation}
where
\begin{equation}\label{eq:defk2}
\kappa_2=\sqrt{\dfrac{9(X_1f)^2}{16f^2}+\dfrac{9f^2}{4}}.
\end{equation}
 We have the following lemma.

\begin{lemma}\label{lem-nablax2xi}
The function $\kappa_2$ and the vector field $\xi$ satisfy
\begin{itemize}
\item[(a)] $X_2\kappa_2=0$;\vspace{2mm}
\item[(b)] $\overline{\nabla}_{X_2}\xi=-\kappa_2\,X_2$;\vspace{2mm}
\item[(c)] $4 (X_1\kappa_2)/\kappa_2=3 (X_1f)/f$;\vspace{2mm}
\item[(d)] $\overline{\nabla}_{X_1}\xi=0$.
\end{itemize}
\end{lemma}
\begin{proof}
From $X_2f=0$ and $[X_1,X_2]=3 (X_1f)X_2/(4f)$, if follows that
$$
X_2X_1f=X_1X_2f-[X_1,X_2]f=0.
$$
Since $\kappa_2$ depends only on $f$ and $X_1f$, (a) follows.
To prove (b), using (a) and \eqref{eq:connection-r3}, we have
\begin{eqnarray*}
\overline{\nabla}_{X_2}\xi&=&\frac{1}{\kappa_2}\overline{\nabla}_{X_2}\left(\dfrac{3(X_1f)}{4f}X_1+\dfrac{3f}{2}\eta\right)\\
&=& \frac{1}{\kappa_2}\left( -\dfrac{9(X_1f)^2}{16f^2} X_2- \frac{9f^2}{4} X_2\right)\\
&=&-\frac{1}{\kappa_2}\, \kappa_2^2\, X_2=-\kappa_2\,X_2\,.
\end{eqnarray*}
To prove (c), first observe that a direct computation gives
$$
4\frac{X_1\kappa_2}{\kappa_2}=\frac{1}{4f^4}\frac{9f^2 (X_1f)(X_1X_1f)-9f (X_1f)^3+36 f^5 (X_1f)}{\dfrac{9(X_1f)^2}{16f^2}+\dfrac{9f^2}{4}}\,.
$$
Then (c) is equivalent to
$$
3\frac{X_1f}{f}=\frac{1}{4f^4}\frac{9f^2 (X_1f)(X_1X_1f)-9f (X_1f)^3+36 f^5 (X_1f)}{\dfrac{9(X_1f)^2}{16f^2}+\dfrac{9f^2}{4}}
$$
which is itself equivalent to 
$$
f(X_1X_1f)-\frac{7}{4}(X_1f)^2+f^4=0\,.
$$
Now, the latter equation is \eqref{eq:f-bi-conservative-local} with $c=0$ (see also \eqref{eq:f-prime-equation}).

We now prove (d). First, from a direct computation, taking into account \eqref{eq:connection-r3}, we have
$$
\overline{\nabla}_{X_1}\xi=\frac{3}{4}\left(X_1\left(\frac{X_1f}{f\kappa_2}\right)+\frac{f^2}{\kappa_2}\right)X_1+\frac{3}{2}\left(X_1\left(\frac{f}{\kappa_2}\right)-\frac{1}{4}\frac{X_1f}{\kappa_2}\right)\eta\,.
$$
We have to show that both components are zero. First
$$
X_1\left(\frac{f}{\kappa_2}\right)-\frac{1}{4}\frac{X_1f}{\kappa_2}=0
$$
if and only if
$$
4 \frac{X_1\kappa_2}{\kappa_2}=3 \frac{X_1f}{f},
$$
which is identity (c).
Similarly, using (c),
$$
X_1\left(\frac{X_1f}{f\kappa_2}\right)+\frac{f^2}{\kappa_2}=0
$$
if and only if
$$
f(X_1X_1f)-\frac{7}{4}(X_1f)^2+f^4=0,
$$
which is identity \eqref{eq:f-prime-equation}.
\end{proof}

\begin{remark}
It is useful to observe that, from Lemma~\ref{lem-nablax2xi}, (a)-(b), the integral curves of the vector field
$X_2$ are circles in $\r^3$ with curvature $\kappa_2$.
\end{remark}
We are now in the right position to state the main result of this section.

\begin{theorem}[See also Proposition~3.1 in \cite{HV95}]\label{teo:r3general}
Let $M^2$ be a biconservative surface in $\r^3$ with $f(p)>0$ and $\grad f(p)\neq 0$ for any $p\in M$. Then, locally, $M^2$ is a surface of revolution.
\end{theorem}
\begin{proof}
Let $\gamma$ be an integrable curve of $X_2$ parametrized by arc-length. From Lemma~\ref{lem-nablax2xi}, (a)-(b), $\gamma$  is a circle in $\r^3$ with curvature $\kappa_2$, that can be parametrized by
\begin{equation}\label{eq:gammas}
\gamma(s)=c_0+c_1 \cos(\kappa_2 s)+c_2 \sin(\kappa_2 s),\quad c_0,c_1,c_2\in\r^3
\end{equation}
with
$$
|c_1|=|c_2|=\frac{1}{\kappa_2}\,,\quad \langle c_1,c_2\rangle=0\,.
$$
Let $p_0\in M$ be an arbitrary point and let $\sigma(u)$ be an integral curve of $X_1$ with $\sigma(0)=p_0$. Consider the flow $\phi$
of the vector field $X_2$ near the point $p_0$. Then, for all $u\in(-\delta,\delta)$ and for all $s\in(-\varepsilon,\varepsilon)$,
$$
\phi_{\sigma(u)}(s)=c_0(u)+c_1(u) \cos(\kappa_2(u) s)+c_2(u) \sin(\kappa_2(u) s),
$$
where the vectorial functions $c_0(u),c_1(u),c_2(u)$, which are uniquely determined by their initial conditions, satisfy
$$
\sigma(u)=c_0(u)+c_1(u)\,,\quad|c_1(u)|=|c_2(u)|=\frac{1}{\kappa_2(u)}\,,\quad \langle c_1(u),c_2(u)\rangle=0\,,
$$
while $\kappa_2(u)=\kappa_2(\sigma(u))$.
Thus, locally, the surface can be parametrized by
$$
X(u,s)=\phi_{\sigma(u)}(s)\,.
$$
Now, since $\kappa_2(0)>0$, there exists $\delta'>0$ such that for $u\in(-\delta',\delta')$, we have $\kappa_2(u)>\kappa_2(0)/2$. Then we can reparametrize $X(u,s)$  using the change of parameter
$$(u,s)\to (u,v=\kappa_2(u) s),$$
where $v$ is defined  in a interval which includes $(-\kappa_2(0)\varepsilon/2,\kappa_2(0)\varepsilon/2)$.
With respect to the above change of parameters, the parametrization of the surface becomes
$$
X(u,v)=c_0(u)+\frac{1}{\kappa_2(u)}\left(C_1(u) \cos(v)+C_2(u) \sin(v)\right),
$$
where
$$
C_1(u)=\kappa_2(u) c_1(u)\,, \quad C_2(u)=\kappa_2(u) c_2(u)\,.
$$
Since the integral curves of $X_2$ start (at $v=0$) from $\sigma$, we have
$$
\sigma(u)=X(u,0)=c_0(u)+\frac{1}{\kappa_2(u)} C_1(u).
$$
From this
\begin{equation}\label{eq:Xuv}
X(u,v)=\sigma(u)+\frac{1}{\kappa_2(u)}\big(C_1(u) (\cos v -1)+C_2(u) \sin v \big).
\end{equation}
Using \eqref{eq:gammas} we find
$$
C_2=\kappa_2 c_2=\gamma'(0)=X_2(\gamma(0)),
$$
which implies that $C_2(u)=X_2(\sigma(u))$. Using  \eqref{eq:gammas} again, we get
$$
- \kappa_2^2\, c_1=\gamma''(0)=\kappa_2(\gamma(0))\, \xi(\gamma(0))=\kappa_2(u)\, \xi(\sigma(u)),
$$
which implies that $C_1(u)=-\xi(\sigma(u))$.
Now we shall prove that $C_1(u)$ and $C_2(u)$ are, in fact, constant vectors. Indeed, taking into account Lemma~\ref{lem-nablax2xi},(d),
$$
\frac{d C_1}{du}=-\overline{\nabla}_{\sigma'}\xi=-\overline{\nabla}_{X_1}\xi=0.
$$
Moreover, using \eqref{eq:connection-r3},
$$
\frac{d C_2}{du}=\overline{\nabla}_{\sigma'}X_2=\overline{\nabla}_{X_1}X_2=0.
$$
Thus the image of the parametrization \eqref{eq:Xuv} is given by a 1-parameter family of circles passing through the points of $\sigma(u)$ lying in affine planes parallel to the space spanned by $C_1$ and $C_2$.

To finish the proof we need to show that the curve of the centers of the circles is a line orthogonal to $C_1$ and $C_2$.
The parametrization \eqref{eq:Xuv} can be written as
$$
X(u,v)=\beta(u)+\frac{1}{\kappa_2(u)}\big(C_1\cos v +C_2 \sin v \big),
$$
where
$$
\beta(u)=\sigma(u)-\frac{C_1}{\kappa_2(u)}
$$
is the curve of the centers. Let show that $\beta$ is a line. For this we prove that $\beta'\wedge \beta''=0$.
Since
 $$
 \sigma''(u)=-\frac{f(u)}{2}\, \eta(\sigma(u)),
 $$
 where $f(u)=f(\sigma(u))$ and $X_1\wedge X_2=\eta$, we have
\begin{eqnarray*}
\beta'\wedge \beta''&=& \left(\sigma'-\left(\frac{1}{\kappa_2}\right)' C_1\right)\wedge\left(\sigma''-\left(\frac{1}{\kappa_2}\right)''C_1\right)\\
&=& -\frac{f}{2} X_1\wedge \eta + \left(\frac{1}{\kappa_2}\right)'' X_1\wedge\xi-\frac{f}{2} \left(\frac{1}{\kappa_2}\right)' \xi\wedge\eta\\
\rm(using\;  \eqref{eq-defxi})&=&\left(\frac{f}{2}-3\frac{f}{2} \left(\frac{1}{\kappa_2}\right)'' \left(\frac{1}{\kappa_2}\right)+\frac{3}{4} \frac{X_1f}{2}\left(\frac{1}{\kappa_2}\right)\left(\frac{1}{\kappa_2}\right)'  \right)X_2\,.
\end{eqnarray*}
Now, replacing \eqref{eq:defk2} in
$$
\left(\frac{f}{2}-3\frac{f}{2} \left(\frac{1}{\kappa_2}\right)'' \left(\frac{1}{\kappa_2}\right)+\frac{3}{4} \frac{X_1f}{2}\left(\frac{1}{\kappa_2}\right)\left(\frac{1}{\kappa_2}\right)'  \right)
$$
and using the identities \eqref{eq:f-prime-equation} and Lemma~\ref{lem-nablax2xi}, (c), we find zero.

Finally, $\beta'$ is clearly orthogonal to $C_2$ and
\begin{eqnarray*}
\langle\beta',C_1\rangle&=&\langle X_1, C_1\rangle-  \left(\frac{1}{\kappa_2}\right)'\\
&=& -\langle X_1, \xi\rangle-  \left(\frac{1}{\kappa_2}\right)'\\
\rm{(using\; \eqref{eq-defxi})}&=& -\frac{1}{\kappa_2}\left(\frac{3}{4} \frac{X_1f}{f}-\frac{\kappa_2'}{\kappa_2}\right)\\
\rm{(using\; Lemma~\ref{lem-nablax2xi}\, (c))}&=&0\,.
\end{eqnarray*}
\end{proof}

%
\section{Biconservative surfaces in $\s^3$}

In this section we consider biconservative surfaces in $3$-dimensional sphere $\s^3$. We assume that the surface is not CMC and thus we can choose $f$ to be positive and $\grad f\neq 0$ at any point of the surface. We define the local orthonormal frame $\{X_1,X_2\}$ as
in \eqref{eq:def-x1-x2} and we look at $\s^3$ as a submanifold of $\r^4$. With this in mind and denoting by $\nabla, \nabla^{\s^3}$ and $\overline{\nabla}$ the connections of $M$, $\s^3$ and $\r^4$, respectively, we have at a point $\mathbf r \in M \subset \s^3 \subset \r^4$

\begin{equation}\label{eq:connection-s3}
\left\{\begin{array}{ll}
\nabla^{\s^3}_{X_1}X_1=-\dfrac{f}{2}\eta,&\quad \nabla^{\s^3}_{X_1}X_2=0,\\
&\\
\nabla^{\s^3}_{X_2}X_1=-\dfrac{3(X_1f)}{4f}X_2,&\quad \nabla^{\s^3}_{X_2}X_2=\dfrac{3(X_1f)}{4f}X_1+\dfrac{3f}{2}\eta,\\
\end{array}
\right.
\end{equation}
and
\begin{equation}\label{eq:connection-r4}
\left\{\begin{array}{ll}
\overline{\nabla}_{X_1}X_1=-\dfrac{f}{2}\eta-{{\mathbf r}},&\quad \overline{\nabla}_{X_1}X_2=0,\\
&\\
\overline{\nabla}_{X_2}X_1=-\dfrac{3(X_1f)}{4f}X_2,&\quad \overline{\nabla}_{X_2}X_2=\dfrac{3(X_1f)}{4f}X_1+\dfrac{3f}{2}\eta-{{\mathbf r}},\\
&\\
 \overline{\nabla}_{X_1}\eta=\dfrac{f}{2}X_1,&\quad \overline{\nabla}_{X_2}\eta=-\dfrac{3f}{2}X_2\,,
\end{array}
\right.
\end{equation}
where $\eta$ is a unit vector field  normal  to the surface $M$ in $\s^3$.
Put
\begin{equation}\label{eq-defxis3}
\kappa_2\, \xi=\dfrac{3(X_1f)}{4f}X_1+\dfrac{3f}{2}\eta-{\mathbf r}=\overline{\nabla}_{X_2}X_2
\end{equation}
where
\begin{equation}\label{eq:defk2s3}
\kappa_2=\sqrt{\dfrac{9(X_1f)^2}{16f^2}+\dfrac{9f^2}{4}+1}.
\end{equation}
 We have the following analogue of Lemma~\ref{lem-nablax2xi}.

\begin{lemma}\label{lem-nablax2xi-s3}
The function $\kappa_2$ and the vector field $\xi$ satisfy
\begin{itemize}
\item[(a)] $X_2\kappa_2=0$;\vspace{2mm}
\item[(b)] $\overline{\nabla}_{X_2}\xi=-\kappa_2\,X_2$;\vspace{2mm}
\item[(c)] $4 (X_1\kappa_2)/\kappa_2=3 (X_1f)/f$;\vspace{2mm}
\item[(d)] $\overline{\nabla}_{X_1}\xi=0$.
\end{itemize}
\end{lemma}

Now, let $M^2$ be a biconservative surface in $\s^3$ with $f>0$ and $\grad f\neq 0$ at any point. Then, using the same argument as in the proof of Theorem~\ref{teo:r3general}, we find that, locally, $M^2\subset\r^4$ can be parametrized by
\begin{equation}\label{eq:xuvs3}
X(u,v)=\sigma(u)+\frac{1}{\kappa_2(u)}\big(C_1(u) (\cos v-1) +C_2(u) \sin v \big),
\end{equation}
where $\sigma(u)$ is an integral curve of $X_1$, $\kappa_2(u)=\kappa_2(\sigma(u))$ is the curvature of the integral curves
of $X_2$, which are circles in $\r^4$, and $C_1,C_2$ are two vector functions such that $|C_1|=|C_2|=1$ and $\langle C_1,C_2\rangle=0$. Moreover,
\begin{equation}\label{eq:c1-c2-s3-value}
C_1(u)=-\xi(\sigma(u)),\quad C_2(u)=X_2(\sigma(u))\,.
\end{equation}
Further, it is easy to see that $C_1$ and $C_2$ are constant vectors. Then, it is clear from \eqref{eq:xuvs3} that locally the surface $M^2$ is  given by a family of circles of $\r^4$, passing through the curve $\sigma$,  and belonging to a pencil of planes which are parallel to the linear space spanned by $C_1$ and $C_2$. Now, these circles must be the intersection of the pencil with the sphere $\s^3$. Let $G$ be the 1-parameter group of isometries of $\r^4$ generated by the Killing vector field
$$
T=\langle  {\mathbf r}, C_2\rangle C_1+\langle  {\mathbf r}, C_1\rangle C_2\,.
$$ 
Then $G$ acts also on $\s^3$ by isometries and it can be identified with the group $SO(2)$. Since the orbits
of $G$ are circles of $\s^3$ we deduce that $X(u,v)$, in \eqref{eq:xuvs3}, describes an $SO(2)$ invariant surface of $\s^3$ obtained by the action of $G$ on the curve $\sigma$. Moreover, we can give the following explicit construction.
\begin{theorem}\label{teo:class-bicons-s3}
Let $M^2$ be a biconservative surface in $\s^3$ with $f>0$ and $\grad f\neq 0$ at any point. Then, 
 locally, $M^2\subset\r^4$ can be parametrized by
\begin{equation}\label{eq:xuvs3-s2}
X_C(u,v)=\sigma(u)+\frac{4}{3 \sqrt{C} k(u)^{3/4}}\big(C_1 (\cos v-1) +C_2 \sin v \big),
\end{equation}
where  $C$ is a positive constant of integration, $C_1,C_2\in\r^4$ are two constant orthonormal vectors such that
\begin{equation}\label{eq:c1c2s3}
\langle \sigma(u), C_1\rangle =\frac{4}{3 \sqrt{C} k(u)^{3/4}}\,,\quad \langle \sigma(u), C_2\rangle=0\,,
\end{equation}
while $\sigma=\sigma(u)$ is a curve lying in the totally geodesic $\s^2=\s^3\cap\Pi$ ($\Pi$ the linear hyperspace of $\r^4$ orthogonal to $C_2$), whose geodesic  curvature $k=k(u)$ is a positive non constant solution of
the following ODE
\begin{equation}\label{eq:sigma-k-s2}
k'' k =\frac{7}{4} (k')^2+\frac{4}{3} k^2-4 k^4\,.
\end{equation}
\end{theorem}
\begin{proof}
From \eqref{eq:xuvs3} we know that 
$$
X(u,v)=\sigma(u)+\frac{1}{\kappa_2(u)}\big(C_1 (\cos v-1) +C_2 \sin v \big)\,,
$$
Since 
$$
\langle \sigma(u), C_2\rangle=\langle \sigma(u), X_2(\sigma(u))\rangle=0,
$$
we deduce that $\sigma\subset\Pi$, where $\Pi$ is the hyperplane of $\r^4$  defined by the equation $\langle {\mathbf r},C_2\rangle=0$. Thus $\sigma$ is a curve in $\s^3\cap \Pi=\s^2$, where $\s^2$ is a totally geodesic 2-sphere of $\s^3$. Now, let $k$ denote the geodesic curvature of $\sigma$ in $\s^2$. Then, taking into account \eqref{eq:connection-s3}, we have
$$
\nabla^{\s^2}_{\sigma'}\sigma'=\nabla^{\s^3}_{\sigma'}\sigma'=-\frac{f(u)}{2}\, \eta(\sigma(u))\,,
$$
where $f(u)=f\circ \sigma (u)$. We deduce that $k(u)=|\nabla^{\s^2}_{\sigma'}\sigma'|=f(u)/2$. From \eqref{eq:f-prime-equation}, with $c=1$, we know that $f=f(u)$ is a solution of 
$$
f'' f = \frac{7}{4} (f')^2+ \frac{4}{3}f^2-f^4\,,
$$
which implies that $k=k(u)$ is a solution of \eqref{eq:sigma-k-s2}. To finish we have to compute $\kappa_2(u)$ as a function of $k(u)$. 
First, by a standard argument, we find that \eqref{eq:sigma-k-s2} has the prime integral,
\begin{equation}\label{eq-promeint-ks3}
(k')^2=-\frac{16}{9} k^2-16 k^4+Ck^{7/2}\,,\quad C\in\r,\, C>0\,.
\end{equation}
Substituting \eqref{eq-promeint-ks3} in \eqref{eq:defk2s3} we find
$$
\kappa_2(u)=\frac{3}{4} \sqrt{C} k(u)^{3/4}\,.
$$
Finally, using the value of $C_1$ in \eqref{eq:c1-c2-s3-value} and that of $\xi$ in \eqref{eq-defxis3}, we get
$$
\langle \sigma(u), C_1 \rangle=\langle \sigma(u), -\xi(\sigma(u)) \rangle=\frac{1}{\kappa_2(u)}=\frac{4}{3 \sqrt{C} k(u)^{3/4}}\,.
$$
\end{proof}

\begin{remark}\label{remark-existences3}
Theorem~\ref{teo:class-bicons-s3} asserts that if $M^2$ is a biconservative surface of $\s^3$, then, locally, it is an $SO(2)$-invariant surface whose profile curve $\sigma$ satisfies \eqref{eq:c1c2s3} and \eqref{eq:sigma-k-s2}.  It is worth to show that such a curve exists. 

First, the condition in Theorem~\ref{teo:class-bicons-s3} that $k$ is a positive non constant solution of 
\eqref{eq:sigma-k-s2} is not restrictive. In fact, choosing the initial condition $k(u_0)>0$ and $k'(u_0)>0$,
from Picard's theorem there is a unique solution of \eqref{eq:sigma-k-s2} which is positive and non constant in an open interval containing $u_0$.

Next,  let assume that $C_1=e_3$ and 
$C_2=e_4$, where $\{e_1,\ldots,e_4\}$ is the canonical basis of $\r^4$. Then, using \eqref{eq:c1c2s3}, $\sigma$ can be explicitly described as
\begin{equation}\label{eq:sigma-xy}
\sigma(u)=(x(u),y(u),\frac{4}{3 \sqrt{C}}\, k(u)^{-3/4},0)\,,
\end{equation}
for some functions $x(u)$ and $y(u)$. Since $\sigma$ is parametrized by arc-length and its curvature must be the given function $k$ (i.e. $\sigma''=-k\,\eta-{\mathbf r}$),
the functions  $x=x(u)$ and $y=y(u)$ must satisfy the system
\begin{equation}\label{eq:sigma-sistem-s3-1}
\begin{cases}
x^2+y^2+\dfrac{16}{9 C}\, k^{-3/2}=1\vspace{3mm}\\
(x')^2+(y')^2+\dfrac{16}{9 C}\,\left(\left( k^{-3/4}\right)'\right)^2=1\vspace{3mm}\\
(x'')^2+(y'')^2+\dfrac{16}{9 C}\,\left(\left( k^{-3/4}\right)''\right)^2=1+k^2\,.
\end{cases}
\end{equation}
Taking the derivative and using \eqref{eq:sigma-k-s2}-\eqref{eq-promeint-ks3}, system~\eqref{eq:sigma-sistem-s3-1} becomes
\begin{equation}\label{eq:sigma-sistem-s3-2}
\begin{cases}
x^2+y^2+\dfrac{16}{9 C}\, k^{-3/2}=1\vspace{3mm}\\
(x')^2+(y')^2=\dfrac{16}{9 C}\,(1+9k^2)\, k^{-3/2}\vspace{3mm}\\
(x'')^2+(y'')^2+\dfrac{16}{9 C}\,(1-3k^2)^2\, k^{-3/2}=1+k^2\,.
\end{cases}
\end{equation}
Now, since $k'\neq 0$, we can locally invert the function $k=k(u)$ and write $u=u(k)$. Then System~\eqref{eq:sigma-sistem-s3-2} becomes
\begin{equation}\label{eq:sigma-sistem-s3-3}
\begin{cases}
x^2+y^2+\dfrac{16}{9 C}\, k^{-3/2}=1\vspace{3mm}\\
(k')^2\left(\dfrac{dx}{dk}\right)^2+(k')^2\left(\dfrac{dy}{dk}\right)^2=\dfrac{16}{9 C}\,(1+9k^2)\, k^{-3/2}\vspace{3mm}\\
\left(\dfrac{d^2x}{dk^2}(k')^2+\dfrac{dx}{dk}k''\right)^2+\left(\dfrac{d^2y}{dk^2}(k')^2+\dfrac{dy}{dk}k''\right)^2+\dfrac{16}{9 C}\dfrac{(1-3k^2)^2}{k^{3/2}}=1+k^2\,,
\end{cases}
\end{equation}
where, according to \eqref{eq-promeint-ks3},
$$
(k')^2=-\frac{16}{9} k^2-16 k^4+Ck^{7/2}\,,\quad k''=-\frac{16}{9} k-32 k^3+\frac{7}{4}Ck^{5/2}\,.
$$
From the first equation of \eqref{eq:sigma-sistem-s3-3}, we get
$$
y(k)=\pm\sqrt{1-x(k)^2-\dfrac{16}{9 C}\, k^{-3/2}}\,,
$$
that substituted in the second gives
\begin{eqnarray}\label{eq:dxdks3}
\dfrac{dx}{dk}&=&\dfrac{12 x(k)}{k (9 C
   k^{3/2}-16)}\nonumber\\
   &&\pm\dfrac{36
   \sqrt{-9 C k^{3/2} x(k)^2+9 C k^{3/2}-16}}{(9 C
   k^{3/2}-16) \sqrt{9 C k^{3/2}-144 k^2-16}}\,.
\end{eqnarray}
We note that $dx/dk\neq 0$. In fact, if it were zero, from \eqref{eq:dxdks3},  we should have $x(k)=\pm 3k/\sqrt{1+9k^2}$ which is not constant.
Taking the derivative of  \eqref{eq:dxdks3} with respect to $k$ and replacing in it the value  $dx/dk$ given in \eqref{eq:dxdks3} we find that $d^2x/dk^2$ depends only on $x(k)$ and $k$. In the same way we find that  
$dy/dk$ and $d^2y/dk^2$ depend only on $x(k)$ and $k$. 
Finally, substituting in the third equation of system~\eqref{eq:sigma-sistem-s3-3} the values of 
$dx/dk$, $dy/dk$, $d^2x/dk^2$, $d^2y/dk^2$, $k'$ and $k''$  we find an identity. This means that the solution $x(k)$ of \eqref{eq:dxdks3} and the corresponding $y(k)$ give a curve $\sigma$, as described in \eqref{eq:sigma-xy}, which satisfies all the desired conditions.

Now, although we could not find an explicit solution of \eqref{eq:sigma-k-s2}, which would give the curvature of the profile curve $\sigma$,
using Mathematica we were able to plot a numerical solution as shown in Figure~\ref{fig-plot-k-s3}.

\begin{figure}[h!]
\begin{pspicture}(-5,-.5)(5,5.5)
\psset{xunit=1cm,yunit=1cm,linewidth=.03,arrowscale=2}
\put(-3,.5){\includegraphics[width=5.5cm]{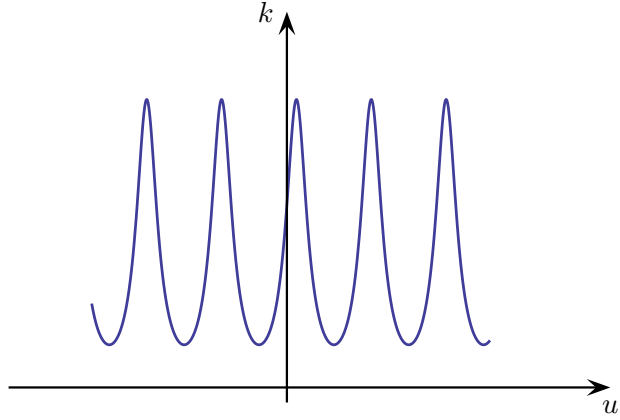}}
\psline{->}(-4,0)(4,0)
\psline{->}(-.3,-.2)(-.3,5)
\uput[270](4,0){$u$}
\uput[180](-.3,5){$k$}
\end{pspicture}
\caption{\label{fig-plot-k-s3}  Plot of a numerical solution of \eqref{eq:sigma-k-s2}
with $k(0)=1$ and $k'(0)=1$. The constant of integration is, in this case, $C=169/9$.}

\end{figure}

%
\end{remark}

\section{Biconservative surfaces in the hyperbolic space}

Let $\l^4$ be the 4-dimensional Lorentz-Minkowski space, that is, the real vector space $\r^4$ endowed with the Lorentzian metric tensor $\langle,\rangle$ given by
$$
\langle,\rangle= dx_1^2 +dx_2^2 +dx_3^2 -dx_4^2,
$$
where $(x_1,x_2,x_3,x_4)$ are the canonical coordinates of $\r^4$. The 3-dimensional unitary hyperbolic
space is given as the following hyperquadric of $\l^4$,
$$
\h^3 = \{(x_1,x_2,x_3,x_4)\in\l^4\,:\, x_1^2 +x_2^2 +x_3^2 -x_4^2 = -1,\, x_4 > 0\}.
$$
As it is well known, the induced metric on $\h^3$ from $\l^4$ is Riemannian  with constant sectional curvature $-1$. In this section we shall use this model of the hyperbolic space.
For convenience we shall recall that, if $X,Y$ are tangent vector fields to $\h^3$, then 
$$
\overline{\nabla}_{X}Y=\nabla^{\h^3}_XY+\langle X,Y\rangle {\mathbf r}
$$
where $\overline{\nabla}$ is the connection on $\l^4$, $\nabla^{\h^3}$ is that of $\h^3$, while ${\mathbf r}$ is the position vector of a point $\mathbf r \in M \subset \h^3 \subset \l^4$.

Let $M^2$ be a  biconservative surface in the $3$-dimensional hyperbolic space $\h^3$. We assume that the surface is not CMC and thus we can choose $f$ to be positive and $\grad f\neq 0$ at any point of the surface. We define again the local orthonormal frame $\{X_1,X_2\}$ as
in \eqref{eq:def-x1-x2}.  We have 

\begin{equation}\label{eq:connection-h3}
\left\{\begin{array}{ll}
\nabla^{\h^3}_{X_1}X_1=-\dfrac{f}{2}\eta,&\quad \nabla^{\h^3}_{X_1}X_2=0,\\
&\\
\nabla^{\h^3}_{X_2}X_1=-\dfrac{3(X_1f)}{4f}X_2,&\quad \nabla^{\h^3}_{X_2}X_2=\dfrac{3(X_1f)}{4f}X_1+\dfrac{3f}{2}\eta,\\
\end{array}
\right.
\end{equation}
and
\begin{equation}\label{eq:connection-r4lorentz}
\left\{\begin{array}{ll}
\overline{\nabla}_{X_1}X_1=-\dfrac{f}{2}\eta+{{\mathbf r}},&\quad \overline{\nabla}_{X_1}X_2=0,\\
&\\
\overline{\nabla}_{X_2}X_1=-\dfrac{3(X_1f)}{4f}X_2,&\quad \overline{\nabla}_{X_2}X_2=\dfrac{3(X_1f)}{4f}X_1+\dfrac{3f}{2}\eta+{{\mathbf r}},\\
&\\
 \overline{\nabla}_{X_1}\eta=\dfrac{f}{2}X_1,&\quad \overline{\nabla}_{X_2}\eta=-\dfrac{3f}{2}X_2\,,
\end{array}
\right.
\end{equation}
where $\eta$ is  a unit  vector field normal to the surface $M$ tangent to $\h^3$.
Put
\begin{equation}\label{eq-defxih3}
\kappa_2\, \xi=\overline{\nabla}_{X_2}X_2=\dfrac{3(X_1f)}{4f}X_1+\dfrac{3f}{2}\eta+{\mathbf r}
\end{equation}
where
\begin{equation}\label{eq:defk2h3}
\kappa_2=\sqrt{\left|\dfrac{9(X_1f)^2}{16f^2}+\dfrac{9f^2}{4}-1\right|}.
\end{equation}

Differently from the case of surfaces in $\r^3$ or in $\s^3$, in this case  the quantity
$$
W=\dfrac{9(X_1f)^2}{16f^2}+\dfrac{9f^2}{4}-1=\dfrac{9|\grad f|^2}{16f^2}+\dfrac{9f^2}{4}-1
$$
can take both positive and negative values. 
Taking this in consideration, we have the following analogue of Lemma~\ref{lem-nablax2xi}.

\begin{lemma}\label{lem-nablax2xi-h3}
The function $\kappa_2$ and the vector field $\xi$ satisfy
\begin{itemize}
\item[(a)] $X_2\kappa_2=0$;\vspace{2mm}
\item[(b)] $\overline{\nabla}_{X_2}\xi=-\varepsilon \kappa_2\,X_2$;\vspace{2mm}
\item[(c)] $4 (X_1\kappa_2)/\kappa_2=3 (X_1f)/f$;\vspace{2mm}
\item[(d)] $\overline{\nabla}_{X_1}\xi=0$,
\end{itemize}
where $\varepsilon$ is $1$ when $W>0$ and is $-1$ when $W<0$.
\end{lemma}

As in the case of biconservative surfaces in $\s^3$, we can give the following explicit construction.
\begin{theorem}\label{teo:class-bicons-h3}
Let $M^2$ be a biconservative surface in $\h^3$ with $f>0$ and $\grad f\neq 0$ at any point. Then, 
 locally, $M^2\subset \l^4$ can be parametrized by:
 \begin{enumerate}
 \item[(a)] if $W>0$,
\begin{equation}\label{eq:xuvs3-h31}
X_C(u,v)=\sigma(u)+\frac{4}{3 \sqrt{C} k(u)^{3/4}}\big(C_1 (\cos v-1) +C_2 \sin v \big),
\end{equation}
where  $C$ is a positive constant of integration, $C_1,C_2\in\l^4$ are two constant vectors such that
\begin{equation}\label{eq:c1c2h3}
\langle C_i,C_j\rangle =\delta_{ij}\,,\quad \langle \sigma(u), C_1\rangle =\frac{4}{3 \sqrt{C} k(u)^{3/4}}\,,\quad \langle \sigma(u), C_2\rangle=0\,,
\end{equation}
while $\sigma=\sigma(u)$ is a curve lying in the totally geodesic $\h^2=\h^3\cap\Pi$ ($\Pi$ the linear hyperspace of $\l^4$ defined by $\langle{\mathbf r},C_2\rangle=0$), whose geodesic  curvature $k=k(u)$ is a positive non constant solution of
the following ODE
\begin{equation}\label{eq:sigma-k-h2}
k'' k =\frac{7}{4} (k')^2-\frac{4}{3} k^2-4 k^4\,;
\end{equation}
 \item[(b)]  if $W<0$,
 \begin{equation}\label{eq:xuvh3-h32}
X_C(u,v)=\sigma(u)+\frac{4}{3 \sqrt{-C} k(u)^{3/4}}\big(C_1 (e^v-1) +C_2 (e^{-v} -1)\big),
\end{equation}
where  $C$ is a negative constant of integration, $C_1,C_2\in\l^4$ are two constant vectors such that
\begin{equation}\label{eq:c1c2h3bis}
\langle C_i,C_i\rangle =0\,,\quad\langle C_1,C_2\rangle =-1\,,\quad \langle \sigma(u), C_1\rangle =\langle \sigma(u), C_2\rangle=-\frac{2\sqrt{2}}{3 \sqrt{-C} k(u)^{3/4}}\,,
\end{equation}
while $\sigma=\sigma(u)$ is a curve lying in the totally geodesic $\h^2=\h^3\cap\Pi$ ($\Pi$ the linear hyperspace of $\l^4$ defined by $\langle{\mathbf r},C_1-C_2\rangle=0$), whose geodesic  curvature $k=k(u)$ is a positive non constant solution of \eqref{eq:sigma-k-h2}.
\end{enumerate}
\end{theorem}
\begin{proof}
(a). In this case $W>0$. Define the local orthonormal frame $\{X_1,X_2\}$ as
in \eqref{eq:def-x1-x2}. Let $\gamma(s)$ be an integral curve of $X_2$ parametrized by arc-length. Then
from 
$$
\gamma''(s)=\overline{\nabla}_{\gamma'}\gamma'=\kappa_2(s) \xi(s)
$$
and
$$
\gamma'''(s)=\overline{\nabla}_{\gamma'}\gamma''=-\kappa_2^2(s) \gamma'(s)
$$
it follows that the parametrization $\gamma(s)$ satisfies the following ODE
$$
\gamma'''+\kappa_2^2\gamma'=0\,.
$$
Then, as we have proceeded in the proof of Theorem~\ref{teo:r3general}, we find that, locally, 
$M^2\subset \l^4$ can be parametrized by
\begin{equation}\label{eq:xuvh3}
X(u,v)=\sigma(u)+\frac{1}{\kappa_2(u)}\big(C_1 (\cos v-1) +C_2 \sin v \big),
\end{equation}
where $\sigma(u)$ is and integral curve of $X_1$, $\kappa_2(u)=\kappa_2(\sigma(u))$ is the curvature of the integral curves
of $X_2$ and $C_1,C_2\in\l^4$ are two constant vectors such that
\begin{equation}\label{eq:c1-c2-h3-value}
\langle C_i,C_j\rangle =\delta_{ij}\,,\quad C_1=-\xi(\sigma(u))\,,\quad C_2=X_2(\sigma(u))\,.
\end{equation}
Since 
$$
\langle \sigma(u), C_2\rangle=\langle \sigma(u), X_2(\sigma(u))\rangle=0,
$$
we deduce that $\sigma\subset\Pi$, where $\Pi$ is the hyperspace of $\l^4$  defined by the equation $\langle {\mathbf r},C_2\rangle=0$. Thus $\sigma$ is a curve in $\h^3\cap \Pi=\h^2$, where $\h^2$ is totally geodesic  in $\h^3$. Now, let $k=k(u)$ denote the geodesic curvature of $\sigma$ in $\h^2$. Then, as in the proof of Theorem~\ref{teo:class-bicons-s3}, we find that $k$ is a solution of \eqref{eq:sigma-k-h2}.
In order to conclude, we have to compute $\kappa_2(u)$ as a function of $k(u)$. 
First, by a standard argument, we find that \eqref{eq:sigma-k-h2} has the prime integral
\begin{equation}\label{eq-promeint-kh3}
(k')^2=\frac{16}{9} k^2-16 k^4+Ck^{7/2}\,,\quad C\in\r,\, C>0\,.
\end{equation}
Substituting \eqref{eq-promeint-kh3} in \eqref{eq:defk2h3} and recalling that $k(u)=|\nabla^{\h^3}_{\sigma'}\sigma'|=f(u)/2$, we find
$$
\kappa_2(u)=\frac{3}{4} \sqrt{C} k(u)^{3/4}\,.
$$
Finally, by using the value of $C_1$ in \eqref{eq:c1-c2-h3-value} and that of $\xi$ in \eqref{eq-defxih3}, we get
$$
\langle \sigma(u), C_1 \rangle=\langle \sigma(u), -\xi(\sigma(u)) \rangle=\frac{1}{\kappa_2(u)}=\frac{4}{3 \sqrt{C} k(u)^{3/4}}\,.
$$
(b). In this case $W<0$ and the curve  $\gamma(s)$ satisfies the following ODE
$$
\gamma'''- \kappa_2^2\gamma'=0.
$$
Thus $\gamma(s)=c_o +c_1\, e^{\kappa_2 s} +c_2\, e^{-\kappa_2 s}$, where, since $\langle\gamma',\gamma'\rangle=1$, $c_1$ and $c_2$ are vectorial functions such that $\langle c_1, c_1\rangle =\langle c_2, c_2\rangle=0$ and $\langle c_1, c_2\rangle =-1/(2 \kappa_2^2)$.
It follows that, locally, $M^2\subset \l^4$ can be parametrized by
$$
X(u,s)=c_0(u)+c_1(u)\, e^{\kappa_2(u) s} +c_2(u)\, e^{-\kappa_2(u) s},
$$
where $\kappa_2(u)=\kappa_2(\sigma(u))$, $\sigma=\sigma(u)$ being an integral curve of $X_1$.
Now, if we perform the change of variables $v=\kappa_2(u) s$ and use the condition $X(u,0)=\sigma(u)$, we obtain that the parametrization of  $M^2$ in $\l^4$ is

\begin{equation}\label{eq:xuvh3bis}
X(u,v)=\sigma(u)+\frac{1}{\sqrt{2} \kappa_2(u)}\big(C_1 (e^v-1) +C_2(e^{-v}-1) \big),
\end{equation}
where  $C_1,C_2\in\r^4$ are two constant vectors such that
$$
\langle C_i,C_i\rangle =0\,,\quad\langle C_1,C_2\rangle =-1\,,\quad C_1+C_2=\sqrt{2}\,\xi(\sigma(u))\,,\quad C_1-C_2=\sqrt{2}\,X_2(\sigma(u)).
$$
Since 
$$
\langle \sigma(u), C_1-C_2\rangle=\sqrt{2}\langle \sigma(u), X_2(\sigma(u))\rangle=0
$$
we deduce that $\sigma\subset\Pi$, where $\Pi$ is the hyperspace of $\l^4$  defined by the equation $\langle {\mathbf r},C_1-C_2\rangle=0$. 
Thus $\sigma$ is a curve in $\h^3\cap \Pi=\h^2$, where $\h^2$ is totally geodesic  in $\h^3$. Now, let $k(u)$ denote the geodesic curvature 
of $\sigma(u)$ in $\h^2$. Then $k=k(u)$ is a solution of \eqref{eq:sigma-k-h2} and, in this case,  we find the same prime integral
 \eqref{eq-promeint-kh3} but with the constant $C<0$.
Next, as we have done in case (a), we get the value of $\kappa_2(u)$ as a function of $k(u)$ as well as $\langle \sigma(u), C_1\rangle$  and $\langle \sigma(u), C_2\rangle$ as indicated in 
\eqref{eq:c1c2h3bis}.
\end{proof}

\begin{remark}
If we assume that 
 $C_1=e_2$ and 
$C_2=e_1$, where $\{e_1,\ldots,e_4\}$ is the canonical basis of $\l^4$, using an argument as in Remark~\ref{remark-existences3}, 
we can check that the curve $\sigma(u)$  in Theorem~\ref{teo:class-bicons-h3} (a) must be of the form
$$
\sigma(u)=(0,\frac{4}{3 \sqrt{C}}\, k(u)^{-3/4},x(u),y(u))\,,
$$
for some functions $x(u)$ and $y(u)$ which are solution of the system
$$
\begin{cases}
x^2-y^2+\dfrac{16}{9 C}\, k^{-3/2}=-1\vspace{3mm}\\
(x')^2-(y')^2=\dfrac{16}{9 C}\,(9k^2-1)\, k^{-3/2}\vspace{3mm}\\
(x'')^2-(y'')^2+\dfrac{16}{9 C}\,(1+3k^2)^2\, k^{-3/2}=k^2-1\,.
\end{cases}
$$
By a direct computation one can show that this system has a solution.  

For the curve $\sigma(u)$ in Theorem~\ref{teo:class-bicons-h3} (b) we have that, choosing $C_1=e_1+e_4$ and  $C_2=e_2+e_4$, 
$$
\sigma(u)=(y(u)-\frac{\sqrt{2}}{2 \kappa_2(u)},y(u) -\frac{\sqrt{2}}{2 \kappa_2(u)},x(u),y(u))\,,
$$
where, in this case, $x(u)$ and $y(u)$ are solution of the system
$$
\begin{cases}
2 \left(y-\frac{\sqrt{2}}{2 \kappa_2}\right)^2 +x^2-y^2=-1\vspace{3mm}\\
2 \left(\left(y-\frac{\sqrt{2}}{2 \kappa_2}\right)'\right)^2 +(x')^2-(y')^2=1\vspace{3mm}\\
2 \left(\left(y-\frac{\sqrt{2}}{2 \kappa_2}\right)''\right)^2 +(x'')^2-(y'')^2=k^2-1\,.
\end{cases}
$$
Again, using the same machineries as in Remark~\ref{remark-existences3}, we can check that this system has a solution.

Moreover, also in this case, as we have noticed in Remark~\ref{remark-existences3}, we can plot a numerical solution of
\eqref{eq:sigma-k-h2}
as shown in Figure~\ref{fig-plot-k-h3}.

\begin{figure}[h!]
\begin{pspicture}(-5,-.5)(5,4.5)
\psset{xunit=1cm,yunit=1cm,linewidth=.03,arrowscale=2}
\put(-3,-.1){\includegraphics[width=5.5cm]{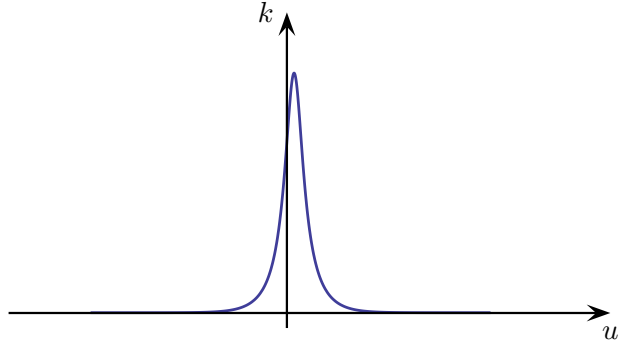}}
\psline{->}(-4,0)(4,0)
\psline{->}(-.3,-.2)(-.3,4)
\uput[270](4,0){$u$}
\uput[180](-.3,4){$k$}
\end{pspicture}
\caption{\label{fig-plot-k-h3} Plot of  a numerical solution of \eqref{eq:sigma-k-h2}
with $k(0)=1$ and $k'(0)=1$ and integration constant $C=137/9$. Choosing $k(0)=1/4$ and $k'(0)=1/5$
we obtain a negative integration constant $C=-248/225$ (thus a solution to the case (b) of Theorem~\ref{teo:class-bicons-h3})  
but the qualitative behavior of $k$ is similar to the case $C>0$.}

\end{figure}

%

\end{remark}

\begin{remark}
We have the following geometric interpretation of the surfaces described in Theorem~\ref{teo:class-bicons-h3} (a). 
As we have already observed, choosing $C_1=e_2$ and 
$C_2=e_1$, where $\{e_1,\ldots,e_4\}$ is the canonical basis of $\l^4$,  the curve $\sigma(u)$  is  of the form
$$
\sigma(u)=(0,\frac{1}{\kappa_2(u)},x(u),y(u))\,,
$$
and the corresponding biconservative surface is parametrized by
$$
X(u,v)=(\frac{1}{\kappa_2(u)} \sin v, \frac{1}{\kappa_2(u)} \cos v, x(u), y(u))\,.
$$
Therefore, the surface  is clearly given by the action, on the curve $\sigma$, of the group of isometries of $\l^4$ 
which leaves the plane $P^2$ generated by $e_3$ and $e_4$ fixed. These surfaces, following the terminology 
given by do Carmo and Dajczer (see \cite{DoCDa83}), are called rotational surfaces of spherical type. In fact,  
the metric of $\l^4$ restricted on  $P^2$ is Lorentzian and when this happens, as described in \cite[pag. 688]{DoCDa83}, the orbits are circles.
\end{remark}

\end{document}